\documentclass[12pt]{amsart}
\usepackage{amsmath,amsthm,amsfonts,amssymb,mathrsfs}
\date{\today}

\usepackage{color}

\newtheorem{theorem}{Theorem}
\newtheorem{corollary}[theorem]{Corollary}

\theoremstyle{definition}

\newtheorem{remark}[theorem]{Remark}

\usepackage{hyperref}

\begin{document}

\title[On $\mathscr{H}$-complete topological semilattices]{On $\mathscr{H}$-complete topological semilattices}

\author{S.\ Bardyla and O.\ Gutik}
\address{Department of Mathematics, Ivan Franko Lviv National
University, Universytetska 1, Lviv, 79000, Ukraine}
\email{sbardyla@yahoo.com, o\underline{\hskip5pt}\,gutik@franko.lviv.ua,
ovgutik@yahoo.com}



\begin{abstract}
In the paper we describe the structure of $\mathscr{A\!H}$-completions and $\mathscr{H}$-completions of the discrete semilattices $(\mathbb{N},\min)$ and $(\mathbb{N},\max)$. We give an example of an $\mathscr{H}$-complete topological semilattice which is not $\mathscr{A\!H}$-complete. Also we construct an $\mathscr{H}$-complete topological semilattice of cardinality $\lambda$ which has $2^\lambda$ many open-and-closed continuous homomorphic images which are not $\mathscr{H}$-complete topological semilattices. The constructed examples give a negative answer to Question~17 from \cite{Stepp1975}.
\end{abstract}

\maketitle

In this paper all topological spaces will be assumed to be Hausdorff. We shall follow the terminology of~\cite{CHK,
Engelking1989}, and \cite{GHKLMS}. For a subset $A$ of a topological space $X$ by $\operatorname{cl}_X(A)$ we denote the closure of $A$ in $X$. A filter $\mathscr{F}$ on a set $S$ is called \emph{free} if $\bigcap\mathscr{F}=\varnothing$.

A {\em semilattice} is a set endowed with a commutative idempotent associative operation. If $E$ is a semilattice, then the semilattice operation on $E$ determines the partial order $\leqslant$ on $E$:
\begin{equation*}
e\leqslant f \qquad\text{if and only if} \qquad ef=fe=e.
\end{equation*}
This order is called {\em natural}. An element $e$ of a semilattice $E$ is called {\em minimal} ({\em maximal\/}) if $f\leqslant e$ ($e\geqslant f$) implies $f=e$ for $f\in E$. A semilattice $E$ is said to be {\em linearly ordered} or a \emph{chain} if the natural order on $E$ is linear.

If $S$ is a topological space equipped with a continuous semigroup operation then $S$ is called a {\em topological semigroup}. A {\em topological semilattice} is a topological semigroup which is algebraically a semilattice.

Let $\mathscr{T\!S}$ be a category whose objects are topological semigroups and morphisms are homomorphisms between topological semigroups. A topological semigroup $X\in\operatorname{Ob}\mathscr{T\!S}$ is called {\em $\mathscr{T\!S}$-complete} if for each object $Y\in\operatorname{Ob}\mathscr{T\!S}$ and a morphism $f\colon X\rightarrow Y$ of the category $\mathscr{T\!S}$ the image $f(X)$ is closed in $Y$.

By a {\em $\mathscr{T\!S}$-completion} of a topological semigroup $X$ we understand any $\mathscr{T\!S}$-complete topological semigroup $\tilde X\in\operatorname{Ob}\mathscr{T\!S}$ containing $X$ as a dense subsemigroup. A $\mathscr{T\!S}$-comp\-le\-tion $\widetilde{X}$ of $X$ is called {\em universal} if each continuous homomorphism $h\colon X\rightarrow Y$ to a $\mathscr{T\!S}$-complete topological semigroup $Y\in\operatorname{Ob}\mathscr{T\!S}$ extends to a continuous homomorphism $\tilde h\colon\widetilde{X}\rightarrow Y$.

It is well-known that for the category $\mathscr{T\!G}$ of topological groups and their continuous homomorphisms, each object $G\in\operatorname{Ob}\mathscr{T\!G}$ has a $\mathscr{T\!G}$-completion and each $\mathscr{T\!G}$-completion of $G$ is universal \cite{Raikov1946}.

In the category of topological semigroups the situation is totally different. We show this on the example of the discrete topological semigroups $(\mathbb{N},\min)$ and $(\mathbb{N},\max)$. We shall study $\mathscr{H}$-completions and $\mathscr{A\!H}$-completions of discrete topological semigroup $(\mathbb{N},\min)$ and $(\mathbb{N},\max)$ in the category $\mathscr{A\!H}$ (resp. $\mathscr{H}$)  whose objects are Hausdorff topological semigroups and morphisms are continuous homomorphisms (resp. isomorphic topological embeddings) between topological semigroups.

The notion of $\mathscr{H}$-completion was introduced by Stepp in \cite{Stepp1969}, where he showed that for each locally compact topological semigroup $S$ there exists an $\mathscr{H}$-complete topological semigroup $T$ which contains $S$ as a dense subsemigroup.

Stepp \cite{Stepp1975} proved that a discrete semilattice $E$ is $\mathscr{H}$-complete if and only if any maximal chain in $E$ is finite. In \cite{GutikPavlyk2003} Gutik and Pavlyk remarked that a topological semilattice is $\mathscr{H}$-complete ($\mathscr{A\!H}$-complete) if and only if it is ($\mathscr{A\!H}$-complete) as a topological semigroup.
In~\cite{GutikRepovs2008} Gutik and Repov\v{s} studied properties of
linearly ordered $\mathscr{H}$-complete topological semilattices and
proved the following characterization theorem:

\begin{theorem}[{\cite[Theorem~2]{GutikRepovs2008}}]\label{theorem-1}
A linearly ordered topological semilattice $E$ is $\mathscr{H}$-complete if and only if the following conditions hold:
\begin{itemize}
    \item[$(i)$] $E$ is complete;
    \item[$(ii)$] $x=\sup A$ for $A={\downarrow}A$
    implies $x\in\operatorname{cl}_EA$; and
    \item[$(iii)$] $x=\inf B$ for $B={\uparrow}B$
    implies $x\in\operatorname{cl}_EB$.
\end{itemize}
\end{theorem}

Also, in \cite{GutikRepovs2008} Gutik and Repov\v{s} proved that each linearly ordered $\mathscr{H}$-complete topological semilattice is $\mathscr{A\!H}$-complete and showed that every linearly ordered semilattice is a dense subsemilattice of an $\mathscr{H}$-complete topological semilattice. In \cite{ChuchmanGutik2007} Chuchman and Gutik proved that any $\mathscr{H}$-complete locally compact topological semilattice and any $\mathscr{H}$-complete topological weakly U-semilattice contain minimal idempotents.

In \cite[Question~17]{Stepp1975} Stepp asked the following question: \emph{Is each $\mathscr{H}$-complete topological semilattice $\mathscr{A\!H}$-complete?} In the present paper we answer this Stepp's question in the negative by constructing an example of an $\mathscr{H}$-complete topological semilattice which is not  $\mathscr{A\!H}$-complete. Also we construct an $\mathscr{H}$-complete topological semilattice of arbitrary infinite cardinality $\lambda$ which has $2^\lambda$ many open-and-closed continuous homomorphic images which are not $\mathscr{H}$-complete topological semilattices.

Let $\mathbb{N}$ denote the set of positive integers. For
each free filter $\mathscr{F}$ on $\mathbb{N}$ consider the topological space $\mathbb{N}_{\mathscr{F}}=\mathbb{N}\cup \{\mathscr{F}\}$ in which all points $x\in\mathbb{N}$ are isolated while the sets $F\cup\{\mathscr{F}\}$, $F\in\mathscr{F}$, form a neighbourhood base at the unique non-isolated point $\mathscr{F}$.

The semilattice operation $\min$ (resp., $\max$) of $\mathbb{N}$ extends to a continuous semilattice operation $\max$ $\min$ (resp., $\max$) on $\mathbb{N}_{\mathscr{F}}$ such that $\min\{n,\mathscr{F}\}=\min\{\mathscr{F},n\}=n$ and $\min\{\mathscr{F},\mathscr{F}\}=\mathscr{F}$ (resp., $\max\{n,\mathscr{F}\}=\max\{\mathscr{F},n\}=\mathscr{F}=\max\{\mathscr{F},\mathscr{F}\}$) for all $n\in\mathbb{N}$. By $\mathbb{N}_{\mathscr{F},\min}$ (resp., $\mathbb{N}_{\mathscr{F},\max}$) we shall denote the topological space $\mathbb{N}_{\mathscr{F}}$ with the semilattice operation $\min$ (resp., $\max$). Simple verifications show that $\mathbb{N}_{\mathscr{F},\min}$ and $\mathbb{N}_{\mathscr{F},\max}$ are topological semilattices.

\begin{theorem}\label{theorem-2}
\begin{enumerate}
\item[$(i)$] For each free filter $\mathscr{F}$ on $\mathbb{N}$ the topological semilattices $\mathbb{N}_{\mathscr{F},\min}$ and $\mathbb{N}_{\mathscr{F},\max}$ are $\mathscr{A\!H}$-complete.
\item[$(ii)$] Each $\mathscr{H}$-completion of the discrete semilattice $(\mathbb{N},\min)$ (resp., $(\mathbb{N},\max)$) is topologically isomorphic to the topological semilattice $\mathbb{N}_{\mathscr{F},\min}$ (resp., $\mathbb{N}_{\mathscr{F},\max}$) for some free filter $\mathscr{F}$ on $\mathbb{N}$.
\item[$(iii)$] The topological semilattice $(\mathbb{N},\min)$ (resp., $(\mathbb{N},\max)$) has no universal $\mathscr{A\!H}$-completion.
\end{enumerate}
\end{theorem}

\begin{proof}[Proof]
$(i)$ By Theorem~\ref{theorem-1}, we have that the topological semilattices $\mathbb{N}_{\mathscr{F},\min}$ and $\mathbb{N}_{\mathscr{F},\max}$ are $\mathscr{H}$-complete. Since $\mathbb{N}_{\mathscr{F},\min}$ and $\mathbb{N}_{\mathscr{F},\max}$ are linearly ordered semilattices, Theorem~3 of \cite{GutikRepovs2008} implies that the topological semilattices $\mathbb{N}_{\mathscr{F},\min}$ and $\mathbb{N}_{\mathscr{F},\max}$ are $\mathscr{A\!H}$-complete.

\medskip
$(ii)$ We shall prove the statement for the semilattice $(\mathbb{N},\min)$. In the case of $(\mathbb{N},\max)$ the proof is similar.

Let $S$ be an $\mathcal H$-complete topological semilattice containing $(\mathbb N,\min)$ as a dense subsemilattice.
Since the closure of a linearly ordered subsemilattice in a Hausdorff topological semigroup is a linearly ordered topological semilattice (see \cite[Corollary~19]{GutikPavlyk2003} and \cite[Lemma~1]{GutikRepovs2008}), we conclude that $S$ is linearly ordered and $S\setminus \mathbb N$ is a singleton $\{a\}$.  Then since $(\mathbb{N},\min)$ is a dense subsemilattice of $S$, the continuity of the semilattice operation in $S$ implies that $a\cdot a=a$ and $a\cdot n=n\cdot a=n$ for any $n\in\mathbb{N}$. Let $\mathscr{B}(a)$ be the filter of neighborhoods of the point $a$ in $S$. This filter induces the free filter $\mathscr{F}=\{F\subset \mathbb N:F\cup\{a\}\in \mathcal B(a)\}$. Then we can identify the topological semilattice $S$ with $\mathbb{N}_{\mathscr{F},\min}$ by the topological isomorphism $f\colon S\rightarrow \mathbb{N}_{\mathscr{F},\min}$ such that $f(a)=\mathscr{F}$ and $f(n)=n$ for every $n\in\mathbb{N}$.

\medskip
$(iii)$ Suppose the contrary: there exists a universal $\mathscr{A\!H}$-completion $S$ of the discrete semilattice $(\mathbb{N},\max)$. Then by statement $(ii)$, the semilattice $S$ can be identified with the semilattice $\mathbb{N}_{\mathscr{F},\max}$ for some free filter $\mathscr{F}$ on $\mathbb N$. Let $\mathscr{F}^{\prime}$ be any free filter on $\mathbb{N}$ such that $\mathscr{F}^{\prime}\not\subset\mathscr{F}$. Then the identity embedding $\operatorname{id}_{\mathbb{N}} \colon(\mathbb{N},\max)\rightarrow\mathbb{N}_{\mathscr{F}^{\prime},\max}$ cannot be extend to a continuous homomorphism $h\colon\mathbb{N}_{\mathscr{F},\max} \rightarrow \mathbb{N}_{\mathscr{F}^{\prime},\max}$, witnessing that the $\mathscr{AH}$-completion $S$ of $(\mathbb N,\min)$ is not universal.
\end{proof}

Later on, by $E_2=\{0,1\}$ we denote the discrete topological semilattice with the semilattice operation $\min$.

\begin{theorem}\label{theorem-3}
Let $\mathscr{F}$ be a free filter on $\mathbb{N}$ and $F\in\mathscr{F}$ be a set with infinite complement $\mathbb{N}\setminus F$. Then the following statements hold:
\begin{enumerate}
\item[$(i)$] the closed subsemilattice $E=\left(\mathbb{N}_{\mathscr{F},\min}\times\{0\}\right)\cup \left((\mathbb{N}\setminus F)\times\{1\}\right)$ of the direct product $\mathbb{N}_{\mathscr{F},\min}\times E_2$ is $\mathscr{H}$-complete;
\item[$(ii)$] the subset $I=\mathbb{N}_{\mathscr{F},\min}\times\{0\}$ is an open-and-closed ideal in $E$, and the quotient semilattice $E/I$ with the quotient topology is discrete and not $\mathscr{H}$-complete;
\item[$(iii)$] the semilattice $E$ is not $\mathscr{A\!H}$-complete.
\end{enumerate}
\end{theorem}

\begin{proof}[Proof]
$(i)$ The definition of the topological semilattice $\mathbb{N}_{\mathscr{F},\min}\times E_2$ implies that $E$ is a closed subsemilattice of $\mathbb{N}_{\mathscr{F},\min}\times E_2$.

Suppose the contrary: the topological semilattice $E$ is not $\mathscr{H}$-complete. Since the closure of a subsemilattice in a topological semigroup is a semilattice (see Corollary~19 of \cite{GutikPavlyk2003}), we conclude that there exists a topological semilattice $S$ which contains $E$ as a dense subsemilattice and $S\setminus E\neq\varnothing$. We fix an arbitrary $a\in S\setminus E$. Then for every open neighbourhood $U(a)$ of the point $a$ in $S$ we have that the set $U(a)\cap E$ is infinite. By Theorem~\ref{theorem-2}, the subspace $\mathbb{N}_{\mathscr{F},\min}\times\{0\}$ of $E$ with the induced semilattice operation from $E$ is an $\mathscr{H}$-complete topological semilattice. Therefore, there exists an open neighbourhood $U(a)$ of the point $a$ in $S$ such that $U(a)\cap E\subseteq (\mathbb{N}\setminus F)\times\{1\}$ and hence the set $U(a)\cap((\mathbb{N}\setminus F)\times\{1\})$ is infinite.

Next we shall show that $a\cdot x=x$ for any $x\in E\setminus\{(\mathscr{F},0)\}$. Since the set $U(x)\cap((\mathbb{N}\setminus F)\times\{1\})$ is infinite, the continuity of the semilattice operation in $E$ implies that $a\cdot x=x$ for any $x\in(\mathbb{N}\setminus F)\times\{1\}$. Now fix any point $y\in\mathbb{N}\times \{0\}\subset E$. By the definition of the semilattice operation on $E$, we can find a point $x_y\in(\mathbb{N}\setminus F)\times\{1\}$ with $x_y\cdot y=y$ and conclude that
\begin{equation*}
a\cdot y=a\cdot(x_y\cdot y)=(a\cdot x_y)\cdot y=x_y\cdot y=y.
\end{equation*}
Since $(\mathscr{F},0)$ is a cluster point of the set $\mathbb N\times\{0\}$, the continuity of the semilattice operation implies that $a\cdot(\mathscr{F},0)=(\mathscr{F},0)$.

Since $W(\mathscr F,0)=(F\cup\{\mathscr{F}\})\times\{0\}$ is a neighborhood of the point $(\mathscr F,0)=a\cdot(\mathscr{F},0)$, the continuity of the semilattice operation yields the existence of neighborhoods $U(a)$ and $V(\mathscr{F},0)$ of the points $a$ and $(\mathscr{F},0)$ in $S$ such that $U(a)\cdot V(\mathscr{F},0)\subset W(\mathscr F,0)$.
Now choose any point $(n,1)\in U(a)\cap \big((\mathbb N\setminus F)\times\{1\}\big)$ and find a point $(m,0)\in V(\mathscr{F},0)$ such that $m\ge n$. Then $$(n,0)=(n,1)\cdot (m,1)\in U(a)\cdot U(\mathscr F,0)\subset W(\mathscr F,0)=(F\cup\{\mathscr F\})\times\{0\},$$ which contradicts the choice of $n\in\mathbb N\setminus F$.
\medskip

$(ii)$ The definition of the semilattice $E$ implies that $I=\mathbb{N}_{\mathscr{F},\min}\times\{0\}$ is an open-and-closed ideal in $E$. Then the quotient semilattice $E/I$ (endowed with the quotient topology) is a discrete topological semilattice, topologically isomorphic to the discrete semilattice $(\mathbb{N},\min)$. By Theorem~\ref{theorem-1}, the semilattice $E/I$ is not $\mathscr{H}$-complete.

\medskip
Statement $(iii)$ follows from statement $(ii)$.
\end{proof}

\begin{corollary}\label{corollary-4}
For a free filter $\mathscr{F}$ on $\mathbb{N}$, each closed subsemilattice of the semilattice $\mathbb{N}_{\mathscr{F},\min}\times E_2$ is $\mathscr{A\!H}$-complete if and only if $\mathscr{F}$ is the filter of cofinite sets on $\mathbb{N}$.
\end{corollary}

\begin{proof}[Proof]
$(\Leftarrow)$ If $\mathscr{F}$ is the filter of cofinite sets on $\mathbb{N}$, then the space $\mathbb{N}_{\mathscr{F},\min}\times E_2$ is compact. Then each closed subset of $\mathbb{N}_{\mathscr{F},\min}\times E_2$ is compact and hence each closed subsemilattice of the semilattice $\mathbb{N}_{\mathscr{F},\min}\times E_2$ is $\mathscr{A\!H}$-complete.

$(\Rightarrow)$ If $\mathscr{F}$ is a free filter on $\mathbb{N}$ containing a set $F\subseteq\mathbb{N}$ with the infinite complement $\mathbb{N}\setminus F$, then $E=\big(\mathbb{N}_{\mathscr{F},\min}\times \{0\}\big)\cup\big((\mathbb N\setminus F)\times\{1\}\big)$ is a closed subsemilattice of the topological semilattice $\mathbb{N}_{\mathscr{F},\min}\times E_2$ and Theorem~\ref{theorem-3} implies that $E$ is not $\mathscr{A\!H}$-complete.
\end{proof}

The proof of the following theorem is similar to the proof of Theorem~\ref{theorem-3} with some simple modifications.

\begin{theorem}\label{theorem-5}
Let $\mathscr{F}$ be a free filter on $\mathbb{N}$ and $F\in\mathscr{F}$ be a set with infinite complement $\mathbb{N}\setminus F$. Then the following assertions hold:
\begin{enumerate}
\item[$(i)$] the closed subsemilattice $E=\left(\mathbb{N}_{\mathscr{F},\max}\times\{0\}\right)\cup \left((\mathbb{N}\setminus F)\times\{1\}\right)$ of the direct product $\mathbb{N}_{\mathscr{F},\max}\times E_2$ is $\mathscr{H}$-complete;
\item[$(ii)$] the subset $I=\mathbb{N}_{\mathscr{F},\max}\times\{0\}$ is an open-and-closed ideal in $E$, and the quotient semilattice $E/I$ with the quotient topology is discrete and not $\mathscr{H}$-complete;
\item[$(iii)$] the semilattice $E$ is not $\mathscr{A\!H}$-complete.
\end{enumerate}
\end{theorem}

The proof of the following corollary is similar to the proof of Corollary~\ref{corollary-4} and follows from Theorem~\ref{theorem-5}.

\begin{corollary}\label{corollary-6}
For a free filter $\mathscr{F}$ on $\mathbb{N}$, each closed subsemilattice of the semilattice $\mathbb{N}_{\mathscr{F},\max}\times E_2$ is $\mathscr{A\!H}$-complete if and only if $\mathscr{F}$ is the filter of cofinite sets on $\mathbb{N}$.
\end{corollary}

We remark that Theorems~\ref{theorem-3} and~\ref{theorem-5} give negative answers on Question~17 from \cite{Stepp1975}.

Also, Theorems~\ref{theorem-3} and~\ref{theorem-5} imply the following corollary:

\begin{corollary}\label{corollary-7}
There exists a countable locally compact $\mathscr{H}$-complete topological semilattice $E$ with an open-and-closed ideal $I$ such that $I$ is an $\mathscr{A\!H}$-complete semilattice and the Rees quotient semigroup $E/I$ with the quotient topology is not $\mathscr{H}$-complete.
\end{corollary}

\begin{remark}\label{remark-8}{\rm
A Hausdorff partially ordered space $X$ is called \emph{$\mathscr{H}$-complete} if $X$ is a closed subspace of every Hausdorff partially ordered space in which it is contained \cite{GutikPagonRepovs2010}. A linearly ordered topological semilattice $E$ is $\mathscr{H}$-complete if and only if $E$ is an $\mathscr{H}$-complete partially ordered space \cite{GutikPagonRepovs2010}. In \cite{Yokoyama??} Yokoyama showed that a partially ordered space $X$ without an infinite antichain is an $\mathscr{H}$-complete partially ordered space if and only if $X$ is a directed complete and down-complete poset such that  $\sup L$ and $\inf L$ are contained in the closure of $L$ for any nonempty chain $L$ in $X$. Theorems~\ref{theorem-3} and~\ref{theorem-5} imply that there exists an $\mathscr{H}$-complete topological semilattice without an infinite antichain which is not an $\mathscr{H}$-complete partially ordered space. Also Theorems~\ref{theorem-3} implies that there exists a countable $\mathscr{H}$-complete locally compact topological semilattice $E$ without an infinite antichain which contains a maximal chain $L$ which is not directed complete, and $L$ does not have a maximal element.}
\end{remark}

Let $\lambda$ be any infinite cardinal and let $0\notin\lambda$. On the set $E_\lambda=\{0\}\cup\lambda$ endowed with the discrete topology we define the semilattice operation by the formula:
\begin{equation*}
x\cdot y=
\left\{
  \begin{array}{ll}
    x, & \hbox{if~} x=y;\\
    0, & \hbox{if~} x\neq y.
  \end{array}
\right.
\end{equation*}

\begin{theorem}\label{theorem-11}
 Let $\mathscr{F}$ be a free filter on $\mathbb{N}$ and $F\in\mathscr{F}$ be a set with infinite complement $\mathbb{N}\setminus F$. Then for each infinite cardinal $\lambda$ the following statements hold:
\begin{enumerate}
\item[$(i)$] the closed subsemilattice $E=\left(\mathbb{N}_{\mathscr{F},\max}\times\{0\}\right)\cup \left((\mathbb{N}\setminus F)\times\lambda\right)$ of the direct product $\mathbb{N}_{\mathscr{F},\max}\times E_{\lambda}$ is $\mathscr{H}$-complete;
\item[$(ii)$] for each subset $\kappa\subset\lambda$ the subset $I_{\kappa}=(\mathbb{N}_{\mathscr{F},\max}\times\{0\})\cup\left((\mathbb{N}
    \setminus F)\times\kappa\right)$ is an open-and-closed ideal in $E$, and the quotient semilattice $E/I_{\kappa}$ with the quotient topology is discrete and not $\mathscr{H}$-complete;
\item[$(iii)$] the semilattice $E$ is not $\mathscr{A\!H}$-complete.
\end{enumerate}
\end{theorem}

\begin{proof}[Proof]
$(i)$ Assuming that the topological semilattice $E$ is not $\mathscr{H}$-complete, find a topological semilattice $T$ containing $E$ as a dense subsemilattice with non-empty complement $T\setminus E$. Fix any element $e\in T\setminus E$. By Theorem~\ref{theorem-1}, the topological semilattice $E^0=\mathbb{N}_{\mathscr{F},\max}\times\{0\}$ is $\mathscr{H}$-complete and hence is closed in $T$. Then there exists an open neighbourhood $U(e)$ of the point $e$ in $T$ such that $U(e)\cap E^0=\varnothing$. By the continuity of the semilattice operation in $T$, there exists an open neighbourhood $V(e)\subseteq U(e)$ of the point $e$ in $T$ such that $V(e)\cdot V(e)\subseteq U(e)$. By Theorem~\ref{theorem-5}, for each $a\in\lambda$, the subsemilattice $E_{a}=\left(\mathbb{N}_{\mathscr{F},\max}\times\{0\}\right)\cup \left((\mathbb{N}\setminus F)\times\{a\}\right)$ of the direct product $\mathbb{N}_{\mathscr{F},\max}\times E_{\lambda}$ is $\mathscr{H}$-complete and hence is closed in $T$. This implies that $V(e)\cap E_{a}\neq\varnothing$  for infinitely many points $a\in\lambda$, and hence $(V(e)\cdot V(e))\cap E^0\neq\varnothing$. This contradicts the choice of the neighbourhood $U(e)$. The obtained contradiction implies that the topological semilattice $E$ is $\mathscr{H}$-complete.

\medskip
$(ii)$ The definition of the semilattice $E$ implies that $I_{\kappa}=(\mathbb{N}_{\mathscr{F},\max}\times\{0\})\cup\left((\mathbb{N}\setminus F)\times\kappa\right)$ is an open-and-closed ideal in $E$. Then we have that the quotient semilattice $E/I_{\kappa}$ with the quotient topology is a discrete topological semilattice. Also, $E/I_{\kappa}$ is topologically isomorphic to the orthogonal sum of $\lambda$ infinitely many of $(\mathbb{N},\max)$ with isolated zero. This implies that the semilattice $E/I_{\kappa}$ is not $\mathscr{H}$-complete.

\medskip
Statement $(iii)$ follows from statement $(ii)$.
\end{proof}

\begin{remark}\label{remark-12}
{\rm The topological semilattices $E$, and $I_{\kappa}$ from Theorem~\ref{theorem-11} are metrizable locally compact spaces for each free countably generated filter $\mathscr{F}$ on $\mathbb{N}$ and any $\kappa\subset\lambda$.}
\end{remark}

\begin{remark} {\rm It can be shown that continuous homomorphisms into the discrete semilattice $\left(\{0,1\},\min\right)$ separate points of the topological semilattices $E$ considered in  Theorems~\ref{theorem-5} and \ref{theorem-11}.}
\end{remark}

Since for each subset $\kappa\subset\lambda$ the natural homomorphism $\pi\colon E\rightarrow E/I_{\kappa}$ is an open-and-closed map,
Theorem~\ref{theorem-11} implies the following corollary:

\begin{corollary}\label{corollary-1.13}
Let $\mathscr{F}$ be a free filter on $\mathbb{N}$ containing a set $F\in\mathscr{F}$ with infinite complement $\mathbb{N}\setminus F$. Then for each infinite cardinal $\lambda$ there exist $2^\lambda$ many continuous open-and-closed surjective homomorphic images of the
topological semilattice $E=\left(\mathbb{N}_{\mathscr{F},\max}\times\{0\}\right)\cup \left((\mathbb{N}\setminus F)\times\lambda\right)\subset \mathbb{N}_{\mathscr{F},\max}\times E_\lambda$,  which are not $\mathscr{H}$-complete.
\end{corollary}


\section*{Acknowledgements}

We acknowledge Taras Banakh for his comments and suggestions.
The authors are also grateful to the referee for several useful comments and suggestions.



\end{document}